\documentclass[12pt]{article}
\usepackage{amssymb,amsthm,amsmath,pb-diagram}
\usepackage{fullpage}

\newtheorem{thm}[equation]{Theorem}
\newtheorem{lemma}[equation]{Lemma}
\newtheorem{prop}[equation]{Proposition}
\newtheorem{cor}[equation]{Corollary}
\theoremstyle{definition}
\newtheorem{remark}[equation]{Remark}
\newtheorem{notation}[equation]{Notation}

\newtheorem{example}[equation]{Example}

\newcommand{\mbb}[1]{\mathbb #1}
\newcommand{\mc}[1]{\mathcal #1}
\newcommand{\oper}[1]{\operatorname{#1}}
\newcommand{\wh}{\widehat}

\newcommand{\GL}{\oper{GL}}
\newcommand{\F}{\mc F}

\newcommand{\Spec}{\oper{Spec}}

\newcommand{\Br}{\oper{Br}}
\newcommand{\per}{\oper{per}}
\newcommand{\ind}{\oper{ind}}
\newcommand{\iso}{\overset{\sim}{\rightarrow}}

\newcommand{\cha}{\oper{char}}
\newcommand{\Mat}{\oper{Mat}}
\title{Weierstrass preparation and algebraic invariants}
\author{David Harbater, Julia Hartmann, and Daniel Krashen} 
\date{September 28, 2011; revised July 16, 2012}
\numberwithin{equation}{section}
\begin{document}
\maketitle

\begin{abstract}
We prove a form of the Weierstrass Preparation Theorem for normal algebraic
curves over complete discrete valuation rings.  While the more traditional algebraic form of Weierstrass Preparation applies
just to the projective line over a base, our version allows more general curves.
This result is then used to obtain applications concerning
the values of $u$-invariants, and on the period-index
problem for division algebras, over fraction fields of complete two-dimensional rings. Our approach uses patching methods and matrix factorization results that can be viewed as analogs of Cartan's Lemma.
\end{abstract}

\section{Introduction} \label{introduction}

The usual algebraic form of the Weierstrass Preparation Theorem (\cite{Bo:CA},
VII.3.9, Proposition~6) says in particular that if $T$ is a complete discrete
valuation ring then every element of $T[[x]]$ can be written as a product
$gu$, where $g \in T[x]$ and where $u$ is a unit in $T[[x]]$.  Thus every
divisor on $\Spec(T[[x]])$ is induced by a divisor on the projective line over
$T$.  There is also an algebraic form of the related Weierstrass Division
Theorem.

Just as the original, analytic form of Weierstrass Preparation and its
companion division theorem are important tools in the theory of several complex variables,
their algebraic forms have been useful in the study of arithmetic curves.  In
particular, versions of these results have been used in connection with
patching and Galois theory; e.g.\ see \cite{Vo}, Theorem~11.3 and Lemma~11.8.

Generalized versions of Weierstrass Preparation, which apply to smooth $T$-curves that need not be the
projective line, were shown in \cite{HH:FP}, Propositions~4.7 and~5.6, in
connection with patching; and this was used in \cite{HHK} to obtain
applications to quadratic forms and central simple algebras (see
Corollaries~4.17 and~5.10 of that paper).  Because of the
smoothness restriction, these forms of Weierstrass Preparation could only be
applied in the context of one-variable function fields over a complete
discretely valued field for which there is a smooth
projective model over $T$. 

The current paper generalizes Weierstrass Preparation further, to $T$-curves
that need not be smooth.  We then apply this
to quadratic forms and central simple algebras over related fields, obtaining
results about $u$-invariants and the period-index problem over fraction fields
of complete two-dimensional rings.  There has been much interest in these
problems; e.g.\ see the papers \cite{COP} and \cite{Hu}, which focused on the
local case in which the residue field is separably closed or finite.
(See also \cite{Hu10}.)
The method we use here is different, and yields applications with more general
residue fields, as well as applications to fraction fields of certain complete
rings that are not local. 

Weierstrass Preparation can be viewed as a factorization assertion for
elements of certain commutative rings.  Our approach here involves first
proving such an assertion for elements of matrix rings, and then specializing
to the $1 \times 1$ case.  
The factorization of matrices as products of matrices over smaller base rings
plays an important role in the framework of patching problems. Patching
permits the 
construction of algebraic objects over rings or fields of functions, given
local data together with agreements on appropriate overfields (e.g.\ see
\cite{Ha:CAPS} and \cite{HH:FP}).   
The matrix factorization result that we prove here is related to \cite{HH:FP},
Theorem~4.6, which was used there to obtain a smooth form of Weierstrass
Preparation;  and is also related to \cite{HHK}, Theorems~3.4 and~3.6.  In
obtaining results here for factorization of elements in the context of
singular $T$-curves, we show that there is an obstruction that does not always
vanish, concerning the reduction graph associated to the closed fiber.  But by
passing to split covers (see \cite{HHK11}, Section~5), we are able to prove a
Weierstrass-type assertion that suffices for our applications. 

The structure of this paper is as follows. 
In Section~\ref{factorization} we explain how patching relates 
to factorization, and then
prove factorization results for the types of rings and fields that arise in
the study of curves over a complete discretely valued field.  This is done first for matrices and then for elements.  The results of Section~\ref{factorization} are applied
in Section~\ref{Weierstrass} to prove our form of Weierstrass
Preparation, and then to obtain corollaries that allow one to pass from local
to global elements modulo an $n$-th power.  Applications are given in
Section~\ref{applications_qf_Br}, which concern quadratic forms
and the $u$-invariant
(Section~\ref{subsec_quadforms}) as well as the period-index problem for division
algebras (Section~\ref{subsec_csa}). 

This work was done in part while the authors were in residence at the Banff International Research Station.  We thank BIRS for its hospitality and resources, which helped advance the research in this paper.

\section{Patching and factorization} \label{factorization}

In this paper, we consider curves over complete discretely valued fields,
using the setup that was introduced in \cite{HH:FP}.  
The approach there used a form of patching to relate structures
on a function field $F$ to structures on certain overfields $F_\xi$ that arise
from geometry, including the realization of the former given the latter.  

A key tool in this study involved the factorization of matrices, to pass from
locally defined objects to more global ones.  This is related to a classical
result, Cartan's Lemma.  While matrix factorization is simplest over fields,
in this paper we will also need to factor matrices over rings, in order to
prove our form of Weierstrass Preparation.  We begin by recalling some
notation and terminology, beginning with the notion of a patching problem. 

For any ring $R$ let $\F(R)$ denote the category of free $R$-modules of finite
rank. Given rings $R \subseteq R_1,R_2 \subseteq R_0$ with $R = R_1 \cap R_2
\subseteq R_0$, a \textit{(free module) patching problem} for $(R,R_1,R_2,R_0)$
is an object in the $2$-fiber product category $\F(R_1) \times_{\F(R_0)}
\F(R_2)$. 
In other words, a patching problem consists of 
a tuple $(M_1,M_2,
M_0;\nu_1,\nu_2)$ of free $R_i$-modules $M_i$ together with isomorphisms
$\nu_i:M_i \otimes_{R_i} R_0 \to M_0$ for $i=1,2$.  In the category of patching problems 
for $(R,R_1,R_2,R_0)$, a morphism $(M_1,M_2,M_0;\nu_1,\nu_2) \to 
(M_1',M_2',M_0';\nu_1',\nu_2')$ consists of $R_i$-module homomorphisms $f_i:M_i \to M_i'$ ($i=0,1,2$) such that $f_0 \circ \nu_i = \nu_i' \circ (f_i \otimes \oper{id}): M_i \otimes_{R_i} R_0 \to M_0'$ for $i=1,2$.

In this situation, there is a functor $\beta: \F(R) \to\F(R_1) \times_{\F(R_0)}
\F(R_2)$ given by base change from $R$ to $R_i$.  That is, $\beta(M) = (M_1,M_2,
M_0;\nu_1,\nu_2)$, where $M_i = M \otimes_R R_i$ for $i=0,1,2$; and where $\nu_i$ is the natural map $(M \otimes_R R_i) \otimes_{R_i} R_0 \to M \otimes_R R_0$.
A \textit{solution} to a patching problem ${\mathcal M}$ is a free $R$-module
$M$ of finite rank whose image under $\beta$ is isomorphic to ${\mathcal M}$.

The following
result appeared in \cite[Prop.~2.1]{Ha:CAPS} and its proof. 

\begin{prop} \label{pp_conditions}
In the above situation, the following conditions are equivalent:
\renewcommand{\theenumi}{\roman{enumi}}
\renewcommand{\labelenumi}{(\roman{enumi})}
\begin{enumerate}
\item The base change functor $\F(R) \to \F(R_1) \times_{\F(R_0)}
\F(R_2)$ is an equivalence of tensor categories;\label{pp1}
\item every free module patching problem for $(R,R_1,R_2,R_0)$ has a
  solution;\label{pp2} 
\item for every $n\ge 1$, every element $A \in \GL_n(R_0)$ can be written as a
  product $A=BC$ with $B \in \GL_n(R_1)$ and $C \in \GL_n(R_2)$.\label{pp3} 
\end{enumerate}
Moreover, under these conditions, the solution to a free module patching
problem as above is given by $M=M_1 \times_{M_0} M_2$, where the fiber product
is taken with respect to the maps $\nu_1,\nu_2$. 
\end{prop}

\begin{proof}  
For the sake of the reader's convenience, we explain the key part of the
proof of the proposition (and of Proposition~2.1 in \cite{Ha:CAPS}), viz.\
that (\ref{pp3}) implies (\ref{pp2}).  For this, given a free module patching
problem defined by modules $M_i$ and isomorphisms $\nu_i$, let $A \in
\GL_n(R_0)$ be the matrix corresponding to the isomorphism
$\nu_2^{-1}\nu_1:M_1 \otimes_{R_1} R_0 \to  M_2 \otimes_{R_2} R_0$, 
with respect to some bases of $M_1,M_2$ over $R_1,R_2$.  Let $B,C$ be as in
(\ref{pp3}).  Adjusting the chosen bases by $B,C$ respectively, the new matrix
for $\nu_2^{-1}\nu_1$ is the identity, and so the new bases have a common
image in $M_0$.  The free $R$-module generated by this basis then gives a
solution to the patching problem.

The converse implication is obtained by reversing this process.  Property
(\ref{pp1}) clearly implies condition (\ref{pp2}), and it is not hard to show
the converse of that implication, using that base change preserves tensor products.
\end{proof}

The geometric situation that we will consider in this paper is described in
the following notation (introduced in \cite{HH:FP}, Section~6; see also \cite{HHK}, Notation~3.3): 

\begin{notation} \label{notn}
Consider a complete discrete valuation ring $T$ with uniformizer $t$, fraction
field $K$, and residue field $k$; a one-variable function field $F$ over $K$,
with a {\em normal model} $\wh X$ of $F$ over $T$ (i.e.\ a projective normal
$T$-scheme with function field $F$); and a finite non-empty 
set of closed points $\mc
P$ of the closed fiber $X$ of $\wh X$, containing all the points where distinct
irreducible components of $X$ meet.  We let $\mc U$ be the set of connected
components of the complement of $\mc P$ in $X$.  For each $P \in \mc P$, we let
$R_P$ be the local ring of $\wh X$ at $P$; we write $\wh R_P$ for its
completion at its maximal ideal; and write $F_P$ for the fraction field of
$\wh R_P$.  For each subset $U$ of $X$ that is contained in an irreducible
component of $X$ and does not meet other components, let $R_U$ be the subring of $F$
consisting of the rational functions on $\wh X$ that are regular on $U$;
write $\wh R_U$ for its $t$-adic completion; and write $F_U$ for the fraction
field of $\wh R_U$.  Thus $F$ is a subfield of each $F_P$ and each $F_U$.
(In the case $U = \{P\}$, note that the field $F_{\{P\}}$ is strictly contained in $F_P$.)
A \textit{branch} of the closed fiber $X$ of $\wh X$ at a point $P \in \mc P$
is a height one prime $\wp$ of $\wh R_P$ that contains $t$. Let $\mc B$ denote
the set of all branches at points $P\in \mc P$. The contraction
of a branch $\wp$ to $R_P$ defines an irreducible component $X_0$ of $X$ (which is the
closure of a unique $U \in \mc U$), and we say that $\wp$ \textit{lies on}
$X_0$.  We let $R_\wp$ be the local ring of $\wh R_P$ at $\wp$; we write $\wh
R_\wp$ for its completion; and write $F_\wp$ for the fraction field of $\wh
R_\wp$.  For a triple $P,U,\wp$ where $\wp$ is a branch at a point $P\in \mc
P$ on the closure of $U\in \mc U$, there are
inclusions of $\wh R_P$ and $\wh R_U$ into $\wh R_{\wp}$. The induced
inclusions of $F_P$ and $F_U$
into $F_{\wp}$ are compatible with the inclusions
$F\hookrightarrow F_P,F_U$.

\end{notation}

In the situation above, we obtain the following factorization result for
matrices.   
Unlike Proposition~\ref{pp_conditions} (which is used in its proof), this
result considers a collection of matrices that are to be factored
simultaneously.
Related results, for fields rather than rings, appeared in \cite{HH:FP}
(Theorems~4.6 and~6.4) and in \cite{HHK} (Theorems~3.4 and 3.6).

\begin{prop} \label{mx_ring_factor} 
Let $\wh X$ be a normal connected projective $T$-curve, with $\mc P, \mc U,
\mc B$ as in Notation~\ref{notn}.
Let $n$ be a positive integer, and suppose that for every branch $\wp \in \mc
B$ we are given an element $A_\wp \in \GL_n(\wh R_\wp)$.  
\renewcommand{\theenumi}{\alph{enumi}}
\renewcommand{\labelenumi}{(\alph{enumi})}
\begin{enumerate}
\item  \label{big_factor}
There exist elements $A_P \in \GL_n(F_P)$ for each $P \in \mc P$,  
and elements $A_U \in \GL_n(\wh R_U)$ for each $U \in \mc U$,
such that for every branch $\wp \in \mc B$ at a point $P \in \mc P$ with $\wp$
lying on the closure of some
$U \in \mc U$, we have $A_\wp = A_P A_U\in \GL_n(F_{\wp})$ with respect
to the natural inclusions $\wh R_U, F_P, \wh R_\wp \hookrightarrow F_\wp$.
\item  \label{small_factor}
There exist elements $A_P' \in \GL_n(\wh R_P)$ for each $P \in \mc P$, and
elements $A_U' \in \GL_n(F_U)$ for each $U \in \mc U$, 
such that for every branch $\wp \in \mc B$ at a point $P \in \mc P$ with $\wp$
lying on the closure of some
$U \in \mc U$, we have $A_\wp = A_P' A_U'\in \GL_n(F_{\wp})$ with respect
to the natural inclusions $F_U, \wh R_P, \wh R_\wp \hookrightarrow F_\wp$.
\end{enumerate}
\end{prop}

\begin{proof} 
\textit{Case 1}: We first consider the case when $\wh X = \mbb P^1_T$ and $\mc
P = \{\infty\}$ consists of the point $P$ at infinity on $\mbb P^1_k$. 

Then 
$\mc U = \{U\}$, where $U$ is the affine line over the residue field $k$ of
$T$.  There is also a single branch $\wp$ at $P$ (on the closure of $U$).   As
in \cite{HH:FP}, 
we write $R_\varnothing$ for the local ring of $\wh X$ at the generic point of
$X$, with completion $\wh R_\varnothing$ and fraction field
$F_\varnothing$. 

For the factorization in (\ref{big_factor}), by~\cite{HH:FP}, Theorem~5.4,
there exist $B_P \in \GL_n(F_P)$ and $A_\varnothing \in \GL_n(F_\varnothing)$
such that $A_\wp = B_P A_\varnothing$.  By~\cite{HH:FP}, Theorem~4.6 and the
comment just after that, there exist $C_P \in \GL_n(F_{\{P\}})$ and $A_U \in
\GL_n(\wh R_U)$ such that $A_\varnothing = C_P A_U$.  Thus $A_\wp = A_P A_U$
where $A_P := B_P C_P \in \GL_n(F_P)$.

For the factorization in (\ref{small_factor}), let $\bar A_\wp$ be the
reduction of $A_\wp$ modulo $t$.  Thus $\bar A_\wp \in
\GL_n\bigl(k((x^{-1}))\bigr)$, where $k((x^{-1}))$ is the fraction field of
the complete local ring $k[[x^{-1}]]$ at the point at infinity on the
projective $k$-line.  By~\cite{Ha:FP}, Lemma~2, we may factor $\bar A_\wp$ as
the product of an invertible matrix over $k[[x^{-1}]]$ and an invertible
matrix over $k(x)$.  
Let $C_P^{-1} \in \GL_n(\wh R_P)$ and 
$C_\varnothing^{-1} \in \GL_n(\wh R_\varnothing)$ be lifts of these matrices.
Thus the matrix 
$A_\wp' := C_P A_\wp C_\varnothing \in \GL_n(\wh R_\wp)$
is congruent to $1$ modulo $t$, as is $A_\wp'^{-1}$.  
By~\cite{HH:FP}, Lemma~5.3, the hypotheses
of~\cite{HH:FP}, Proposition~3.2 are satisfied, with $\wh R_\varnothing, \wh
R_P, \wh R_\wp$ playing the roles of $\wh R_1, \wh R_2, \wh R_0$ there and
taking $M_1 = \wh R_1$ there.  The conclusion of that proposition then says
that $A_\wp'^{-1} = B_\varnothing B_P$ with $B_\varnothing \in 
\Mat_n(\wh R_\varnothing) \cap \GL_n(F_\varnothing)$ and
$B_P \in \GL_n(\wh R_P)$.  Since $A_\wp',B_P \in \GL_n(\wh R_\wp)$, 
the matrix $B_\varnothing$ also lies in that group.  Hence $B_\varnothing$ actually 
lies in $\Mat_n(\wh R_\varnothing) \cap \GL_n(\wh R_\wp) 
= \GL_n(\wh R_\varnothing)$.  
Thus $C_\varnothing B_\varnothing \in \GL_n(\wh R_\varnothing)$.
By~\cite{HH:FP}, Theorem~4.6 and the
comment just after that, there exist 
$D_U \in \GL_n(F_U)$ and $D_P \in \GL_n(\wh R_{\{P\}})$ such that
$C_\varnothing B_\varnothing = D_U D_P$.  
The matrices $A_P' := C_P^{-1}B_P^{-1}D_P^{-1} \in \GL_n(\wh R_P)$ and $A_U' := D_U^{-1} \in \GL_n(F_U)$ then give the desired
factorization $A_\wp = A_P' A_U'$.

\smallskip

\textit{Case 2}: General case.

Choose a finite $T$-morphism $f:\wh X \to \mbb P^1_T$ such that $\mc P =
f^{-1}(\infty)$, using \cite[Proposition~3.3]{HHK11}.  Write $U' = \mbb A_k^1$
and $P' = \infty \in X' := \mbb P_k^1$, and let $\wp'$ be the branch at
infinity on the projective $k$-line.  Also write $F'$ for the function field
of $\mbb P_T^1$.  Let $r = [F:F']$ denote the degree of $f$.  

We claim that the morphism $f$ is flat.  For this, it suffices to show that for every closed point $Q \in X'$, the ring $S_Q \subseteq F$ is free over the local ring $R_Q \subseteq F'$, where $S_Q$ is the subring of $F$ consisting of the rational functions on $\wh X$ that are regular on $f^{-1}(Q)$.  
(That is, $\Spec(S_Q)=\wh X \times_{\mbb P^1_T} \Spec(R_Q)$.)
Now since $\wh X$ is normal, so is the ring $S_Q$.  Thus $S_Q$ satisfies Serre's condition that every prime ideal of codimension two has depth at least two~\cite[pp.~255-256, 462]{Eis}.  
But depth $\le$
codimension (\cite{Eis}, Proposition~18.2).  
So each maximal ideal of $S_Q$ has depth equal to its codimension,
viz.\ two.  
That is, $S_Q$ is a Cohen-Macaulay ring.  Moreover $R_Q$ is a regular local ring, the localizations of $S_Q$ at its maximal ideals all have the same dimension (viz.\ two), and $S_Q$ is finite over $R_Q$.  
Thus \cite[Corollary~18.17]{Eis} applies, and asserts that $S_Q$ is free over $R_Q$.  This proves the claim. 

For part (\ref{big_factor}), consider the localization $\wh R_{P',t}$ of $\wh
R_{P'}$ at the prime ideal $(t)$.  
This is equal to the intersections $F_{P'} \cap \wh R_{\wp'} = F_{P'} \cap \wh R_\varnothing$.  
Both
$\wh R_{U'}$ and $\wh R_{P',t}$ are subrings of $\wh R_{\wp'}$, with intersection  $\wh R_{U'} \cap \wh R_{P',t} = 
\wh R_{U'}  \cap F_{U'} \cap F_{P'} \cap \wh R_\varnothing = \wh R_{U'} \cap F'  = R_{U'}$.  Also, by Case~1, for any matrix 
$A_{\wp'} \in \GL_n(\wh R_{\wp'})$ there exist $A_{P'} \in \GL_n(F_{P'})$ and
$A_{U'} \in \GL_n(\wh R_{U'})$ such that $A_{\wp'} = A_{P'} A_{U'}$.  Since
$A_{\wp'},A_{U'} \in \GL_n(\wh R_{\wp'})$, it follows that the element
$A_{P'}$ lies in $\GL_n(F_{P'}) \cap \GL_n(\wh R_{\wp'})= \GL_n(\wh
R_{P',t})$.  So the quadruple of rings $(R_{U'},\wh R_{U'},\wh R_{P',t},\wh
R_{\wp'})$ satisfies condition~(\ref{pp3}) of Proposition~\ref{pp_conditions},
and hence also condition~(\ref{pp2}) on patching problems. 

Let $V \subset X$ be the union of the sets $U \in \mc U$, and let $R_V$ be the
subring of $F$ consisting of the rational functions on $\wh X$ that are
regular on $V$.  
Thus $V = f^{-1}(U') = \Spec(R_V/tR_V)$.
Since $f$ is flat, the ring $R_V$ is flat over $R_{U'}$; and hence $R_V/tR_V$ is a finite flat module over the ring $R_{U'}/tR_{U'} =
k[x]$.  But this last ring is a principal ideal domain.  Thus 
$R_V/tR_V$ is a free module of rank $r = \deg(f)$ over $R_{U'}/tR_{U'}$; 
and then 
$R_V$ is a free module of rank $r$ over $R_{U'}$ by~\cite[Proposition~II.3.2.5]{Bo:CA}.  Hence $\wh R_{U'} \otimes_{R_{U'}} R_V$, 
$\wh R_{P',t} \otimes_{R_{U'}} R_V$, and $\wh R_{\wp'}  \otimes_{R_{U'}} R_V$
are free modules of rank $r$ over the rings $\wh R_{U'},\wh R_{P',t},\wh
R_{\wp'}$ respectively, with compatible bases.

There are canonical isomorphisms $\wh R_{U'} \otimes_{R_{U'}} R_V \iso \prod_{U \in \mc U} \wh R_U$,
$\wh R_{\wp'} \otimes_{R_{U'}} R_V \iso \prod_{\wp \in \mc B} \wh R_\wp$,
$F_{\wp'} \otimes_{R_{U'}} R_V \iso \prod_{\wp \in \mc B} F_\wp$,
and $F_{P'} \otimes_{R_{U'}} R_V \iso F_{P'} \otimes_{F'} F \iso \prod_{P \in \mc P} F_P$, where the last isomorphism is by~\cite[Lemma~6.2(a)]{HH:FP}.  Since 
$\wh R_{P',t} = F_{P'} \cap \wh R_{\wp'}$, we have an exact sequence
\[0 \to \wh R_{P',t} \to F_{P'} \times \wh R_{\wp'} \to F_{\wp'}\]
of $R_{U'}$-modules,
where the first map is the diagonal inclusion and the second map is given by subtraction.  Tensoring with the free $R_{U'}$-module $R_V$, and using the above isomorphisms, we get an exact sequence 
\[0 \to \wh R_{P',t} \otimes_{R_{U'}} R_V \to \prod_{P \in \mc P} F_P \times \prod_{\wp \in \mc B} \wh R_\wp \to \prod_{\wp \in \mc B} F_\wp.\]
Thus we get canonical identifications 
\[\wh R_{P',t} \otimes_{R_{U'}} R_V = \prod_{P \in \mc P} F_P \,\cap\, \prod_{\wp \in \mc B} \wh R_\wp =
\prod_{P \in \mc P} \bigl(F_P \cap \prod_{\wp \in \mc B_P} \wh R_\wp\bigr),\] 
where
$\mc B_P$ is the set of branches at a given $P \in \mc P$  (i.e., the set of height
one primes in $\wh R_P$ containing $t$).  But for $P \in \mc P$, the intersection $F_P \cap\prod_{\wp \in \mc B_P} \wh R_\wp$ is equal to the semi-localization $\wh R_{P,t}$
of $\wh R_P$ at the set of
branches at $P$.  Thus we also get a canonical identification
$\wh R_{P',t}  \otimes_{R_{U'}} R_V\iso \prod_{P \in \mc P} \wh R_{P,t}$. Since the intersection of $\wh
R_{U'}$ and $\wh R_{P',t}$ in $\wh R_{\wp'}$ is $ R_{U'}$, the intersection of
$\prod_{U\in 
  \mc U} \wh R_U $ and  $\prod_{P\in \mc P} \wh R_{P,t}$ in $\prod_{\wp\in \mc
  B}\wh R_{\wp}$ equals $R_V$ by base change. 

A collection of elements $A_\wp \in \GL_n(\wh R_\wp)$, for $\wp \in \mc B$,
defines an element of $\GL_n(\prod_{\wp\in \mc B} \wh R_\wp)$.  So to prove
part~(\ref{big_factor}) it suffices to show that the quadruple of
rings \[(R_V\,,\,\prod_{U \in \mc U} \wh R_U\,,\,\prod_{P \in \mc P} \wh
R_{P,t}\,,\,\prod_{\wp \in \mc B} \wh R_{\wp})\] 
satisfies condition~(\ref{pp3}) of Proposition~\ref{pp_conditions}, or
equivalently condition~(\ref{pp2}).  Since these rings are respectively free
of rank $r$ over the rings in the quadruple $(R_{U'},\wh R_{U'},\wh
R_{P',t},\wh R_{\wp'})$, a free module patching problem
$(M_1,M_2,M_0;\nu_1,\nu_2)$ of rank $n$ with respect to the first quadruple
induces a free module patching problem of rank $nr$ with respect to the second
quadruple.  As shown above, the latter patching problem has a solution $M$. It
remains to show that the free $R_{U'}$-module $M$ of rank $nr$ is also a free
$R_V$-module of rank $n$ and that $M$ induces $M_1,M_2$ over the rings
$\prod_{U \in \mc U} \wh R_{U}$ and $\prod_{P \in \mc P} \wh R_{P,t}$. 

First observe that $M = M_1 \times_{M_0} M_2$, set-theoretically, by
Proposition~\ref{pp_conditions} above applied to the quadruple $(R_{U'},\wh
R_{U'},\wh R_{P',t},\wh R_{\wp'})$.  Since each $M_i$ is an $R_V$-module and
since the maps $\nu_i:M_i \to M_0$ are $R_V$-module homomorphisms, $M$ is an
$R_V$-module, compatibly.  Since $M$ is a solution to the patching problem
over $F'$, we have identifications $M_1 = \wh R_{U'} \otimes_{R_{U'}}  M = \wh
R_{U'} \otimes_{R_{U'}} R_V \otimes_{R_V} M = {\prod \wh R_U} \otimes_{R_V}
M$. 
That is,  $M$ induces $M_1$ over $\prod_{U \in \mc U} \wh R_{U}$.  The case of
$M_2$ is similar.  
Finally, since $M_1 = \prod \wh R_{U} \otimes M$ is free of
rank $n$ over $\prod \wh R_U$, it follows that $M_1/tM_1 = M/tM$ is free of
rank $n$ over $\prod \wh R_U/t\wh R_U = \prod  R_U/ tR_U = R_V/tR_V$; and thus $M$ is
free of rank $n$ over $R_V$ by~\cite{Bo:CA}, Proposition~II.3.2.5.  
This completes the proof of part~(\ref{big_factor}). 

The proof for (\ref{small_factor}) is similar, but with the roles of $U,P$
reversed.  We replace $\wh R_{P',t}$ with $\wh R_{U',t}$, the localization of
$\wh R_{U'}$ at the prime ideal $(t)$.  This is equal to each of the
intersections $F_{U'} \cap \wh R_{\wp'} = F_{U'} \cap \wh R_\varnothing$.  For
$U \in \mc U$ we similarly replace $\wh R_{P,t}$ with the localization $\wh
R_{U,t}$ of~$\wh R_U$ at its Jacobson radical, which is the unique height one
prime containing $t$.  
We obtain canonical isomorphisms as before.
The given factorization problem yields a free module
patching problem for the quadruple of rings   
\[(R_{\mc P}\,,\,\prod_{U \in \mc U} \wh R_{U,t}\,,\,\prod_{P \in \mc P} \wh
R_P\,,\,\prod_{\wp \in \mc B} \wh R_{\wp}),\]
where $R_{\mc P}$ is the 
subring of $F$ consisting of the rational functions on $\wh X$ that are
regular at the points of $\mc P$.
These rings are respectively free of rank $r$ over the rings in the quadruple
$(R_{P'},\wh R_{U',t},\wh R_{P'},\wh R_{\wp'})$, and are obtained by tensoring
those rings over $R_{P'}$ with $R_{\mc P}$.   
The given free module patching problem with respect to the first quadruple
induces a free module patching problem with respect to the second quadruple,
having a solution $M$ by Case~1 and Proposition~\ref{pp_conditions}, as in~(\ref{big_factor}).   

Proceeding as in the proof of~(\ref{big_factor}), it remains to show that $M$ is free of rank $n$ over $R_{\mc P}$.  
By hypothesis, $\prod \wh R_P \otimes M$ is free of
rank $n$ over $\prod \wh R_P$, and so is flat over $\prod \wh R_P$.  
Here $\wh R_P \otimes M$ is the $\frak m_P$-adic completion of $M_{\{P\}} := R_{\{P\}} \otimes M$ by~\cite[Theorem~III.3.4.3(ii)]{Bo:CA}; and thus $M_{\{P\}}$ is flat over $R_{\{P\}}$ 
by~\cite[Proposition~III.5.4.4]{Bo:CA} (with $A=B$ in the notation there).  Since $R_{\{P\}}$ is local and $M$ is finitely generated, it follows that 
the flat $R_{\{P\}}$-module
$M_{\{P\}}$ is free; moreover its rank is $n$, since this is the case modulo $\frak m_P$ (because $\prod \wh R_P \otimes M$ is free of
rank $n$ over $\prod \wh R_P$).  
Since $R_{\mc P}$ is a semi-local ring whose localizations are the rings 
$R_{\{P\}}$ for $P \in \mc P$, it follows by~\cite[Exercise~4.13]{Eis} (or by~\cite[Lemma~1.4.4]{BH}) that $M$ is free of rank $n$ over $R_{\mc P}$. 
\end{proof} 

\begin{cor}  \label{big_patch}
Suppose that for every $U \in \mc U$ we are given an element $a_U \in
F_U^\times$.  Then there exist $b \in F^\times$ and elements $c_U \in \wh
R_U^\times$ such that $a_U=bc_U \in F_U^\times$ for all $U \in \mc U$. 
\end{cor}

\begin{proof}  For each $U \in \mc U$, let $\eta_U$ be its generic point and
  let $t_U \in F$ be a uniformizer at $\eta_U$.  Choose an affine open subset
  $\Spec(R)$ of $\wh X$ that contains the points $\eta_U$.  Then $R$ is
  Noetherian and integrally closed (since $\wh X$ is normal), and thus is a
  Krull domain by  \cite{Bo:CA}, Corollary to Lemma~1 in Section~VII.1.3.  By
  Theorem~4 of \cite{Bo:CA}, Section~VII.1.6, each $t_U$ defines an essential
  valuation of $R$; and then by \cite{Bo:CA}, Proposition~9 of
  Section~VII.1.5, there exists $s \in F^\times$ whose $t_U$-adic valuation is
  the same as that of $t_U$ for all $U \in \mc U$.   
Replacing each $a_U$ by $s^{-1}a_U$, we may assume that $a_U$ is a $t_U$-adic
unit for every $U$.  In particular, $a_U$ is a unit in $\wh R_\wp$ for every
branch $\wp \in \mc B$ lying on $U$. 

For each branch $\wp \in \mc B$, there is a unique $U \in \mc U$ such that
$\wp$ lies on $U$; let $c_\wp \in F_\wp^\times$ be the image of $a_U$.  
Note that $c_\wp \in \wh R_\wp^\times$ because of the above assumption on $a_U$.
Applying Corollary~\ref{mx_ring_factor}(\ref{big_factor}) to these elements
(viewed as $1 \times 1$ matrices), we obtain elements $c_P \in F_P^\times$ for
all $P \in \mc P$, and 
elements $c_U \in \wh R_U^\times$ for all $U \in \mc U$, such that $c_\wp =
c_P c_U \in F_\wp^\times$ for each branch $\wp$ on $U$ at $P$ (with respect to
the inclusions of $F_P$ and $F_U$ into $F_\wp$).  Set $b_P = c_P \in
F_P^\times$ for every $P \in \mc P$, and set $b_U = a_Uc_U^{-1} \in
F_U^\times$ for each $U \in \mc U$.   
For a branch $\wp \in \mc B$ at $P \in \mc P$ lying on $U \in \mc U$, we have
$b_P = c_P = a_Uc_U^{-1} = b_U$ in $\wh R_\wp^\times$.  Hence the elements
$b_P \in F_P^\times$ (for $P \in \mc P$), $b_U \in F_U^\times$ (for $U \in \mc
U$), and the induced elements of $F_\wp^\times$ together form an element of
the inverse system formed by the groups $F_U^\times$, $F_P^\times$,
$F_\wp^\times$; and so define an element $b$ of the inverse limit, which is
$F^\times$ by~\cite[Proposition~6.3]{HH:FP}. Finally, $a_U=b_Uc_U=bc_U$ with
respect to the above inclusions.   
\end{proof} 

The obvious analog of the above result with the roles of $\mc U$ and $\mc P$
interchanged would say that if we are given an element $a_P \in F_P^\times$
for every $P \in \mc P$, then there exist $b \in F^\times$ and elements $c_P
\in \wh R_P^\times$ such that $a_P=bc_P \in F_P^\times$ for all $P \in \mc P$.
But this assertion is false.   
The reason is that if $\wp, \wp' \in \mc B$ are branches lying on a common $U
\in \mc U$ at points $P,P' \in \mc P$, then the $\wp$-adic valuation
$v_{\wp}(a_P)$ of $a_P=bc_P$ 
would have to equal the $\wp'$-adic valuation $v_{\wp'}(a_{P'})$ of
$a_{P'}=bc_{P'}$; viz.\ both would have to have equal the $t_U$-adic
valuations of $b$ (where $t_U$ is a uniformizer at the generic point $\eta_U$
of $U$, as above).  But under this additional hypothesis on the given family
$\{a_P\}$, the analog holds:

\begin{cor}  \label{small_patch}
Suppose that for every $P \in \mc P$ we are given an element $a_P \in
F_P^\times$.  Suppose also that if $\wp,\wp' \in \mc B$ are branches at $P,P'
\in \mc P$ lying on a common $U \in \mc U$, then
$v_\wp(a_P)=v_{\wp'}(a_{P'})$.  Then  there exist $b \in F^\times$ and
elements $c_P \in \wh R_P^\times$ such that $a_P=bc_P \in F_P^\times$ for all
$P \in \mc P$. 
\end{cor}

\begin{proof}  For each $U \in \mc U$, let $n_U$ be the common 
value of $v_\wp(a_P)$ for all $P \in \mc P$ lying on the closure of $U$ and
all branches $\wp \in \mc B$ at $P$ on $U$.   
By \cite{Bo:CA}, Proposition~9 of Section~VII.1.5, there exists $s \in
F^\times$ whose $t_U$-adic valuation is equal to $n_U$ for all $U$.  Thus
$v_\wp(s)=n_U$ for each branch $\wp \in \mc B$ on $U$. Replacing each $a_P$ by
$s^{-1}a_P$, we may assume that $a_P$ is a $t_U$-adic unit for every $P \in
\mc P$ and $U \in \mc U$ such that $P$ lies on the closure of $U$.  In
particular, $a_P$ is a unit in $\wh R_\wp$ for every branch $\wp \in \mc B$ at
$P$.

Let $c_\wp \in \wh R_\wp^\times$ be the image of $a_P$.  Applying
Corollary~\ref{mx_ring_factor}(\ref{small_factor}), in the $1 \times 1$ case,
to the elements $c_\wp^{-1}$, we obtain elements $c_P^{-1} \in \wh R_P^\times$
for all $P \in \mc P$, and elements $c_U^{-1} \in F_U^\times$ for all $U \in
\mc U$, such that $c_\wp^{-1} = c_P^{-1} c_U^{-1} \in F_\wp^\times$ (or
equivalently, $c_\wp = c_U c_P \in F_\wp^\times$) for each branch $\wp$ on $U$
at $P$, with respect to the inclusions of $F_P$ and $F_U$ into $F_\wp$.   
Set $b_U = c_U \in F_U^\times$ for $U \in \mc U$; and set $b_P = a_Pc_P^{-1}
\in F_P^\times$ for $P \in \mc P$.  Again we have that the elements $b_P \in
F_P^\times$ (for $P \in \mc P$) and $b_U \in F_U^\times$ (for $U \in \mc U$)
induce the same elements in $F_\wp^\times$, and so define an element $b \in
F^\times$.  Finally, $a_P=b_Pc_P=bc_P$ with respect to the above inclusions. 
\end{proof}

\begin{remark}
There are variants of the above two results for $n \times n$ matrices, using
that Corollary~\ref{mx_ring_factor} holds for matrices, and not just field
elements.   
It is convenient to state these variants using notations introduced in the
proof of Corollary~\ref{mx_ring_factor}, viz.\ $\wh R_{U,t}$ for the
localization of $\wh R_U$ at its height one prime containing $t$, and $\wh
R_{P,t}$ for the semi-localization of $\wh R_P$ at its branches, i.e., its
height one primes containing $t$. 
With this notation, the analog of Corollary~\ref{big_patch} then asserts that
if for each $U \in \mc U$ we are given a matrix $A_U \in \GL_n(\wh R_{U,t})$,
then there exist $B \in \GL_n(F)$ and $C_U \in \GL_n(\wh R_U)$ such that
$A_U=BC_U \in \GL_n(\wh R_{U,t})$ for all $U \in \mc U$.  Similarly, the
analog of Corollary~\ref{small_patch} says that if for every $P \in \mc P$ we
are given a matrix $A_P \in \GL_n(\wh R_{P,t})$, then there exist $B \in
\GL_n(F)$ and $C_P \in \GL_n(\wh R_P)$ such that $A_P=BC_P \in \GL_n(\wh
R_{P,t})$ for all $P \in \mc P$.  

The reason that we assume here that the given matrices lie in $\GL_n(\wh
R_{U,t})$ (resp.\ in $\GL_n(\wh R_{P,t})$) rather than simply in $\GL_n(F_U)$
(resp.\ in $\GL_n(F_P)$) is that the proofs of the preceding two results
involved a reduction step, which relied on the fact that every non-zero
element in a discretely valued field is a unit in the valuation ring
multiplied by a power of a given uniformizer.  The analog of this fact does
not hold in the $n \times n$ case. Assuming that the given matrices are
defined over the above smaller rings eliminates the need for that reduction
step. Note that it also eliminates the need for the extra hypothesis in
Corollary~\ref{small_patch} when~$n=1$.
\end{remark}

\section{Weierstrass preparation} \label{Weierstrass}

The factorization results of Section~\ref{factorization}  apply in particular
to the situation in which the elements are given just for some of the $U \in
\mc U$, or just for some of the points $P \in \mc P$.  Namely, we suitably
define the other elements, and then apply the result as stated.  In
this way, we obtain versions of the Weierstrass Preparation Theorem,
extending to the singular case results that had been proven in the smooth case in \cite{HH:FP} (see Remark~\ref{weier_rk}(\ref{weier_smooth_case}) below).

As before, we are in the context of Notation~\ref{notn}.  We have the
following version of the Weierstrass preparation theorem:

\begin{thm} \label{field_Weierstrass}
Let $F$ be a one-variable function field over the fraction field of a complete
discrete valuation ring~$T$, and let $\wh X$ be a normal model for~$F$
over~$T$. Let $\mc P,\mc U, \mc B$ be as in Notation~\ref{notn}.

\renewcommand{\theenumi}{\alph{enumi}}
\renewcommand{\labelenumi}{(\alph{enumi})}
\begin{enumerate} 

\item \label{big_patch_field}
If $U \in \mc U$ and $a \in F_U$, then there exist $b \in F$ and $c \in \wh
R_U^\times$ such that $a=bc$.

\item \label{branch_field}
If $\wp \in \mc B$ and $a \in F_\wp$, then there  exist $b \in F$ and $c \in
\wh R_\wp^\times$ such that $a=bc$.

\item \label{small_patch_field}
Let $P \in \mc P$ and $a \in F_P$.
Suppose that if $\wp,\wp' \in \mc B$ are branches at $P$ lying on a common $U
\in \mc U$, then $v_\wp(a)=v_{\wp'}(a)$.  Then there exist $b \in F$ and $c
\in \wh R_P^\times$ such that $a=bc$.

\end{enumerate}
\end{thm}

\begin{proof} 
We may assume without loss of generality that $a$ is nonzero in each case;
otherwise the assertion is trivially true with $b=0$.

(\ref{big_patch_field})  
Set $a_U=a$, and set $a_{U'}=1$ for each $U' \in \mc U$ other than $U$.  The
assertion is now immediate from Corollary~\ref{big_patch}.

(\ref{branch_field})
Consider the irreducible component $X_0$ of $X$ on which $\wp$ lies, and let
$s \in F$ be a uniformizer at the generic point of $X_0$.  Then $s$ is also a
uniformizer for $\wh R_\wp$.  Since $a \ne 0$ we may write $a = a's^m$ for some
$a' \in \wh R_\wp^\times$ and some $m\in \mbb N$.  Now let $b=s^m$ and $c=a'$.   

(\ref{small_patch_field}) 
Let $U_1,\dots,U_n \in \mc U$ be the elements of $\mc U$ whose closures contain $P$.  
For $i=1,\dots,n$, let $I_i \subset R_P$ be the height one prime of $R_P
\subset F$ corresponding to $U_i$.  By the normality hypothesis, the
Noetherian local ring $R_P$ is a Krull domain.  So for each $i$ there exists
$s_i \in R_P$ whose $I_i$-adic valuation is one, and which does not lie in any
other $I_j$ (\cite{Bo:CA}, Proposition~9 of Section~VII.1.5).   
Thus for each branch $\wp \in \mc B$ at $P$ lying on the closure of 
$U_i$, $v_\wp(s_i)=1$  and $v_\wp(s_j)=0$ for $j \ne i$.

Let
$m_i$ be the $I_i$-adic valuation of $a$.  Thus $v_\wp(a)=m_i$ for all
branches $\wp \in \mc B$ at $P$ lying on the closure of $U_i$, and the element
$a_P := a/\prod s_i^{m_i} \in F_P^\times$  
has $\wp$-adic valuation equal to zero for all $\wp \in \mc B$.  Set 
$a_{P'}=1 \in F_{P'}^\times$ for each $P' \in \mc P$ other than $P$.  These
elements together satisfy the hypotheses of Corollary~\ref{small_patch}, each
having valuation equal to zero at each respective branch.   
Thus $a_P=b'c$ for some $b' \in F^\times$ and $c \in \wh R_P^\times$.  Setting
$b=b'\prod s_i^{m_i} \in F$ yields the assertion. 
\end{proof}   

\begin{remark} \label{weier_rk}
\renewcommand{\theenumi}{\alph{enumi}}
\renewcommand{\labelenumi}{(\alph{enumi})}
\begin{enumerate} 
\item \label{weier_smooth_case}
In the case that the model $\wh X$ of $F$ is smooth over $T$, 
parts~(\ref{big_patch_field}) and~(\ref{small_patch_field}) of Theorem~\ref{field_Weierstrass} follow from results in~\cite{HH:FP}.  Namely, 
part~(\ref{big_patch_field}) is given by Corollary~4.8 of~\cite{HH:FP}.  For part~(\ref{small_patch_field}), by writing $a \in F_P$ as a ratio of elements in $\wh R_P$, we are reduced to the case 
that $a$ lies in $\wh R_P$.  Proposition~5.6 of~\cite{HH:FP} then allows us to write $a$ as the product of elements in $\wh R_P^\times$ and $F_{\{P\}}$.  Applying Corollary~4.8 of~\cite{HH:FP} (with $U = \{P\}$) to the latter element then yields the desired factorization of $a$, using that $\wh R_{\{P\}} \subset \wh R_P$.  

\item \label{weier_special_cases}
The extra hypothesis on $a$ in
Theorem~\ref{field_Weierstrass}(\ref{small_patch_field}) is trivially
satisfied if no two branches $\wp \in \mc B$ at $P$ lie on the same $U \in \mc
U$. In particular, it always holds in the smooth (unibranched) case considered
in the version of Weierstrass Preparation that appeared in \cite{HH:FP}. 
It is also satisfied if $a$ is a unit in $\wh R_\wp$ for each branch $\wp
\in \mc B$ at $P$ (since then the valuations $v_\wp(a)$ are each equal to
zero).  This last condition is equivalent to saying that $a \in \wh
R_{P,t}^\times$, in the notation used in the proof of Case~2 of
Proposition~\ref{mx_ring_factor}.  

\item \label{weier_example}
The hypothesis on $a$ is indeed needed in the general case
of Theorem~\ref{field_Weierstrass}(\ref{small_patch_field}), for a similar
reason that a related hypothesis was needed in Corollary~\ref{small_patch}. 

Namely, if $\wp,\wp' \in \mc B$ are branches at $P$ lying on the closure of 
a common $U \in
\mc U$, and if there is a global element $b \in F$ as asserted, then
$\wp$-adic and $\wp'$-adic valuations of $a$ must each be equal to the
$t_U$-adic valuation of $b$, where $t_U$ is a uniformizer at the generic point
of $U$.  Hence the existence of an element $b \in F$ as in the conclusion of the
theorem implies that $v_\wp(a) = v_{\wp'}(a)$. 
As an explicit example, we could take $\wh X$ to be the cover of the
projective $x$-line over $T$ given by $y^2-x^2(1+x)=t$, and $P$ to be the
point $x=y=0$ on the closed fiber.  (Here we assume the residue characteristic
is not two.) 
Let $z \in \wh R_P$ be a square root of $1+x$.  Then the two branches $\wp, 
\wp'$ at $P$ respectively correspond to $y=\pm xz$ along $t=0$. Both branches
lie on the unique irreducible component of the closed fiber.  Let $a=y-xz \in
F_P$.  Then $v_\wp(a)=1$ but $v_{\wp'}(a)=0$. 
Since these two valuations are unequal, the asserted element $b \in F$ cannot
exist, by the above argument, and the conclusion of
Theorem~\ref{field_Weierstrass}(\ref{small_patch_field}) does not hold for $a
\in F_P$. 
\end{enumerate}
\end{remark}

As a consequence, we obtain the following version of Lemma~4.16 of \cite{HHK} for the case that $\wh X$ is not necessarily smooth.

\begin{cor} \label{reduce_to_patches}
With notation as in Theorem~\ref{field_Weierstrass}, 
let $n>0$ be a natural number not divisible by the characteristic of the residue field of $T$.  
\renewcommand{\theenumi}{\alph{enumi}}
\renewcommand{\labelenumi}{(\alph{enumi})}
\begin{enumerate} 
\item \label{root_big}
For $U \in \mc U$ and $a \in F_U$,
there exist elements $b \in F$ and $c \in F_U^\times$ such that $a=bc^n$.   
\item \label{root_branch}
For $\wp \in \mc B$ and $a \in
F_\wp$, there exist elements $b \in F$ and $c \in F_\wp^\times$
such that $a=bc^n$. 
\item \label{root_small} 
Let $P \in \mc P$ and $a \in F_P$.   
If $v_\wp(a)=v_{\wp'}(a)$ for each pair of branches $\wp,\wp' \in \mc B$ at
$P$ lying on the closure of a common $U \in \mc U$, then there exist elements $b \in F$ and $c \in F_P^\times$ such
that $a=bc^n$.  
\end{enumerate}
\end{cor}

\begin{proof}
Again, all statements are trivially true  if $a=0$, so we assume
that $a\neq 0$ in each case.
For part~(\ref{root_big}) we proceed as in the proof of the analogous case of
Lemma~4.16 of \cite{HHK}, but using
Theorem~\ref{field_Weierstrass}(\ref{big_patch_field}) instead of the global
Weierstrass Preparation Theorem in the smooth case (\cite{HH:FP},
Proposition~4.7).  Namely, write $a=a_1/a_2$ with $a_i \in \wh R_U$ and $a_2
\ne 0$.  By Theorem~\ref{field_Weierstrass}(\ref{big_patch_field}), there
exist $b_i \in F^\times$ and $c_i \in \wh R_U^\times$ such that $a_i=b_ic_i$
for $i=1,2$.  The reduction $\bar c_i \in \wh R_U/t\wh R_U = R_U/tR_U$  
of $c_i$ modulo the uniformizer $t$ of $T$ may be lifted to an element $c_i'
\in R_U \subset F$.  Since $c_i/c_i' \equiv 1$ modulo $t\wh R_U$, 
the element $c_i/c_i'$ has an $n$-th root $c_i'' \in \wh R_U$,
by Hensel's Lemma (\cite{Bo:CA}, Corollary~1 to Theorem~III.4.5.2).
Here
$c:= c_1''/c_2''$ lies in $F_U^\times$ and $b := b_1c_1'/b_2c_2'$ lies in
$F$.  These elements then satisfy $a=bc^n$. 

Parts~(\ref{root_branch}) and~(\ref{root_small}) are proved identically,
except that parts~(\ref{branch_field}) and~(\ref{small_patch_field}) of
Theorem~\ref{field_Weierstrass} are used instead of
part~(\ref{big_patch_field}).  
\end{proof}

The extra hypothesis in
Theorem~\ref{field_Weierstrass}(\ref{small_patch_field}) and
Corollary~\ref{reduce_to_patches} can be dropped if we
allow ourselves to pass to an appropriate finite cover of $\wh X$, as we show
in Proposition~\ref{cover_Weierstrass} and its corollary below.  

First, recall from \cite[Section~5]{HHK11} that a \textit{split cover} of $\wh
X$ is a finite morphism $h:\wh X' \to \wh X$ of normal projective $T$-curves
such that the fiber over every point $P$ of $\wh X$ other than the generic
point consists of a disjoint union of copies of $P$.  Such covers are
automatically \'etale; and for each $\xi \in \mc P \cup \mc U$, the pullback
of $\wh X'$ to $\Spec(F_\xi)$ consists of a disjoint union of finitely many
copies of $\Spec(F_\xi)$.   
Also recall (from \cite[Section~6]{HHK11}) that we may associate to $\wh X$ a
\textit{reduction graph} $\Gamma(\wh X, \mc P)$ whose vertices are in
bijection with $\mc P \cup \mc U$, and whose edges are in bijection with the
set of branches $\mc 
B$.  Namely, the edge associated to $\wp \in \mc B$ connects the vertices
associated to $P \in \mc P$ and $U \in \mc U$ when $\wp$ is a branch at $P$
lying on the closure of $U$.  For each split cover $\pi:\wh X' \to \wh X$,  
taking $\mc P' = \pi^{-1}(\mc P)$,
the associated reduction graphs define a covering space $\Gamma(\wh X',\mc P')
\to \Gamma(\wh X,\mc P)$.

If the set $\mc P$ contains every closed point of $X$ at which $X$ is not unibranched,
then this correspondence is a lattice isomorphism between split covers of $\wh
X$ and connected finite covering spaces of $\Gamma(\wh X, \mc P)$ (Proposition~6.2
of~\cite{HHK11}). In particular, there are no non-trivial split covers of $\wh X$ if the
reduction graph is a tree.

\begin{prop} \label{split_cover_tree}
Let $F$ be a one variable function field over the fraction field of a complete
discrete valuation ring $T$, and let $\wh X$ be a normal model for $F$ over $T$.
Let $\mc P$ be a finite non-empty subset of the closed fiber $X$ that contains
all the points where distinct components of $X$ meet.
Then there is a connected split cover $h:\wh X' \to \wh X$ such that distinct
branches in $h^{-1}(\mc B)$ at a common point of $h^{-1}(\mc P)$ lie on
the closures of different components in $h^{-1}(\mc U)$.  
Moreover if $\Gamma(\wh X,\mc P)$ is not a tree, then this cover can be chosen
to be abelian of arbitrarily high degree.
\end{prop}
 
\begin{proof} 
By enlarging the set $\mc P$, we may assume without loss of generality that it
contains all the points at which $X$ is not unibranched (see Hypothesis~5.4 of
\cite{HHK11}).

The fundamental group of the graph $\Gamma = \Gamma(\wh X,\mc P)$ is a free
abelian group of rank $r\geq 0$, whose abelianization is
$H_1(\Gamma,\mbb Z) \simeq \mbb Z^r$.  Consider the finite set $\Delta$ of
loops in $\Gamma$ that consist of exactly two vertices (corresponding to some
$P \in \mc P$ and $U \in \mc U$) and two edges (corresponding to a choice of 
two distinct branches at $P$ on the closure of $U$).  The elements of $\Delta$ induce
finitely many non-trivial elements of the abelianization $\mbb Z^r$, using
the identification $H_1(\Gamma,\mbb Z) \simeq \mbb Z^r$ and the fact that
$\Gamma$ is a one-dimensional simplicial complex. Choose $n$ sufficiently
large so that none of these elements lies in $n\mbb Z^r$, and let $N$ be the
inverse 
image of $n\mbb Z^r$ in $\pi_1(\Gamma)$.  This is a normal subgroup of finite
index, corresponding to a finite Galois covering space $\Gamma'$ of
$\Gamma$. By construction,
none of the elements of $\Delta$ lift to loops in $\Gamma'$.  But any
loop in $\Gamma'$ that consists of two vertices and two edges must map to a
loop in $\Gamma$ of the same type, since $\Gamma' \to \Gamma$ is a covering
space of bipartite graphs.  So in fact $\Gamma'$ contains no loops of this
type.  
Using Proposition~6.2 of~\cite{HHK11} (which is possible because of the extra
assumption on $\mc P$), there is a split cover $\wh X'$ 
of $\wh X$ whose closed fiber $X'$ gives rise to $\Gamma'$.  Thus $\wh X' \to 
\wh X$ has the desired property.  

For the last assertion, notice that if $\Gamma$ is not a tree for the original
set $\mc P$, then enlarging the set $\mc P$ will give a refined graph which is
also not a tree. Since the cover $\wh X'$ constructed above is Galois with group
$\pi_1(\Gamma)/N = \mbb Z^r/n\mbb Z^r$, it is abelian and its degree $nr$ can be
chosen to be arbitrarily large, by choosing $n$ large.
\end{proof}

\begin{example} \label{Tate}
The simplest non-trivial example of Proposition~\ref{split_cover_tree}
is that of a Tate curve; viz.\ the generic fiber of $\wh X$ is an elliptic
curve and the special fiber is a rational nodal curve.  We let $\mc P$ be the
set consisting of the nodal point. In this case $\Gamma(\wh X,\mc P)$
consists of two vertices connected by two edges; $\wh X$ has non-trivial
cyclic split covers of all degrees $>1$; and distinct branches on any of these
covers $\wh X'$ of $\wh X$ lie on the closures of 
distinct components of the closed fiber.
Cf.~\cite{Saito}, Example~2.7, and \cite{HHK}, Example~4.4.  Also compare
Remark~\ref{weier_rk}(\ref{weier_example}).    
\end{example}

If $\wh X'\to \wh X$ is a split cover and $P'$ is a point on the closed fiber
$X'$ lying over $P\in X$, then the associated fields $F_{P'}'$ and $F_P$ (see
Notation~\ref{notn}) are isomorphic under the natural inclusion (see the
beginning of Section~5
of \cite{HHK11}). Hence we may identify (elements of) these two fields.

\begin{prop} \label{cover_Weierstrass}
Let $F$ be a one-variable function field over the fraction field of a complete
discrete valuation ring $T$, and let $\wh X$ be a normal model for $F$ over $T$.
Then there is a split cover $\wh X' \to  \wh X$, with closed
fiber $X'$ and function field $F'$, such that the following holds:
Given $P \in \mc P$ and $a \in F_P$, then for each $P'\in X'$ lying over $P$, there
exist an element $b \in F'$ and a unit $c \in \wh
R_{P'}'^\times$ satisfying $a=bc\in F_{P'}'= F_P$. 
\end{prop} 

\begin{proof}
By Proposition~\ref{split_cover_tree}, there is a split cover $\wh X' \to  
\wh X$ with the property that distinct branches at any point of $X'$ must lie
on distinct components of $X'$.  Given $a \in F_P=F_{P'}'$, the factorization
$a=bc$ now follows from applying
Theorem~\ref{field_Weierstrass}(\ref{small_patch_field}) to $\wh X'$, 
in the special case in which no two branches $\wp \in \mc B$ at any point $P
\in \mc P$ lie on the closure of the same component of the closed fiber (as noted in
Remark~\ref{weier_rk}(\ref{weier_special_cases})). 
\end{proof}

We may now remove the additional assumption in part~(\ref{root_small}) of
Corollary~\ref{reduce_to_patches} above.

\begin{cor} \label{reduce_to_patches_II}
With notation as in Proposition~\ref{cover_Weierstrass}, let $p \ge 0$ be the
characteristic of the residue field of $T$.    
Then there is a split cover $\wh X' \to \wh X$, with closed
fiber $X'$ and function field $F'$, such that for every $n$ not divisible
by~$p$, the 
following holds:
Given $P\in \mc P$ and $a\in F_P$, then for each $P'\in X'$ lying
over $P$, there exist elements $b\in F'$ and $c\in F_{P'}'^{\times}$ such that
$a=bc^n\in F_{P'}'=F_P$.
\end{cor}

\begin{proof}
The proof is the same as that of part~(\ref{root_small}) of
Corollary~\ref{reduce_to_patches}, using  
Proposition~\ref{cover_Weierstrass} instead of Theorem~\ref{field_Weierstrass}.
\end{proof}

We conclude this section with a result which can be used to show that certain
fields are of the form $F_P$. This will be useful in the applications in the
next section.
\begin{lemma} \label{local-to-global}
Let $k$ be a field, let $T = k[[t]]$, and let $E$ be a
finite separable extension of $k((x,y))$, viewed as a $k$-algebra.
Then there is a connected normal projective $T$-curve $\wh X$ and a closed
point $P$ on $\wh X$ such that $E$ is isomorphic to $F_P$ as a $k$-algebra. 
\end{lemma}

\begin{proof}
Let $A$ be the integral closure of $k[[x,y]]$ in $E$.  Then $A$ is unramified
over the locus of $(y+x^n)$ for some positive integer $n$.  
Embedding $T$ in $k[[x,y]]$ by sending $t$ to $y+x^n$, we may 
identify $k[[x,y]]$ with $k[[x,t]]$; and $A$ is then unramified over the locus
of $t=0$.  We may also identify $k[[x,t]]$ with $\wh R'_{P'}$, where $P'$ is
the point $x=t=0$ on $\wh X' = \mbb P^1_T$, the projective $T$-line with
coordinate $x$ and function field $F'$. 

In the case that $E$ is Galois over $k((x,y)) = k((x,t))$, say with group $G$,
Lemma~5.2 of \cite{HS} asserts that there is a finite generically separable
morphism $\wh X \to \wh X'$ of connected normal $T$-curves that is  
$G$-Galois as a branched cover (i.e.\ whose corresponding function field
extension is $G$-Galois) and whose pullback via $\Spec(k[[x,t]]) \to \wh X'$
is isomorphic to $\Spec(A) \to \Spec(k[[x,t]])$.  Thus there is a unique
closed point $P$ on $\wh X$ that lies over the point $P' \in \wh X'$, and $\wh
R_P$ is isomorphic to $A$ as a $k[[x,t]]$-algebra with $G$-action.  Hence
$F_P$ is isomorphic to $E$ as a $G$-Galois field extension of $k((x,t))$, and
in particular as a $k$-algebra.

The general case may now be reduced to the above Galois case.  Namely, let
$E^*$ be the Galois closure of the finite separable extension $E/k((x,t))$.
Let $G$ be the Galois group of $E^*/k((x,t))$ and let $H \subset G$ be the
Galois group of $E^*$ over $E$.   
By the Galois case as above, there is a $G$-Galois normal branched cover 
$\pi^*:\wh X^* \to \mbb P^1_T$, say with function field $F^*$, with a unique
closed point $P^*$ on $\wh X^*$ lying over $P' \in \wh X'$, such that $E^*$ is
isomorphic to $F^*_{P^*}$ over $F'_{P'} = k((x,t))$.  Let $\wh X = \wh X^*/H$,  
let $\pi:\wh X \to \mbb P^1_T$ be the induced morphism (through which $\pi^*$
factors), 
and let $P \in \wh X$ be the image of $P^* \in \wh X^*$.
Then $E$ is isomorphic to $F_P$ as a $k((x,t))$-algebra, and in particular as
a $k$-algebra. 
\end{proof}

\section{Applications} \label{applications_qf_Br}

Using what was shown in Section~\ref{Weierstrass}, we can
strengthen results that were obtained in \cite{HHK} concerning quadratic forms
and central simple algebras, by dropping smoothness assumptions.   

Our setup is as in Notation~\ref{notn}.  Given such a discrete valuation ring
$T$ and function field $F$, there are many normal models $\wh X$ of $F$ over
$T$, and in particular there are always regular models (see \cite{Abhyankar},
\cite{Lipman}).  But in general there
need not exist a smooth model of $F$ over $T$.  While normality sufficed for
some of the results in \cite{HHK}, others required smoothness, because they
relied on the form of Weierstrass Preparation that appeared in \cite{HH:FP},
which itself had assumed smoothness.  Our more general form of Weierstrass
Preparation above (Theorem~\ref{field_Weierstrass}) yielded  
Corollaries~\ref{reduce_to_patches} and~\ref{reduce_to_patches_II}, and these in turn will permit us to
generalize results of \cite{HHK}.

\subsection{Applications to quadratic forms}\label{subsec_quadforms}

The classical $u$-invariant of a field $k$, defined by Kaplansky, is the
largest dimension $u(k)$ of a quadratic form over $k$ that is anisotropic
(i.e.\ has no non-trivial zeroes).  While $u(k)$ is known for certain fields,
in general it is rather mysterious. 

In order to obtain results about the value of $u(k)$ in 
\cite{HHK}, we introduced a related invariant $u_s(k)$.  By definition, this
is the smallest number $n$ such that  
$u(E) \le n$ for every finite field extension $E/k$, and also such that
$u(E) \leq 2n$ for every finitely
generated field extension $E/k$ of transcendence degree one. (It was pointed
out to us by Karim Becher that the value of this invariant remains unchanged if
the first of these two conditions is dropped.)

In the geometric situation considered in Notation~\ref{notn}, Corollary~4.17
of \cite{HHK} related the $u$-invariant of a field $F_U$ or $F_P$ to the
values of $u$ and $u_s$ on $k$ and its finite extensions.  
Using our generalized Weierstrass Preparation Theorem, we can remove the
smoothness assumption that was needed in that result in \cite{HHK}.  The proof
here parallels the argument in \cite{HHK}, with Lemma~4.16 of \cite{HHK}
replaced by Corollaries~\ref{reduce_to_patches} and~\ref{reduce_to_patches_II} above.

\begin{thm} \label{u-inv_patches}
Let $T$ be a complete discrete valuation ring with uniformizer $t$, whose
residue field $k$ is not of characteristic two.  
Let $\wh X$ be a normal projective $T$-curve with closed
fiber $X$, let $\mc P$ be a non-empty finite set of closed points of $\wh X$
that contains every point at which distinct components of $X$ meet, and let $\mc U$
be the set of components of the complement of $\mc P$ in $X$. 
\renewcommand{\theenumi}{\alph{enumi}}
\renewcommand{\labelenumi}{(\alph{enumi})}
\begin{enumerate}
\item \label{u-inv_upper_bound}
If $\xi \in \mc P \cup \mc U$, then $u(F_\xi) \le 4u_s(k)$.
\item \label{u-inv_lower_bound}
Let $\tilde X$ be the normalization of $X$.
If $P \in \mc P$ is a closed point of $X$, then $u(F_P) \ge 4u(\kappa(\tilde P))$ for any $\tilde P \in \tilde X$ lying over $P$.
If $Q \in U \in \mc U$
is a closed point of $X$, then $u(F_U) \ge 4u(\kappa(\tilde Q))$ for any $\tilde Q \in \tilde X$ lying over $Q$.
\end{enumerate}
\end{thm}

\begin{proof} 
(\ref{u-inv_upper_bound})
By the assumption on the characteristic, any quadratic form $q$ over $F_\xi$ may be diagonalized.  So
given a regular quadratic form $q$ of dimension $n > 4u_s(k)$ over $F_\xi$, we may
assume it has the form $a_1x_1^2 + \cdots + a_nx_n^2$ with $a_i \in
F_\xi^\times$.  

If $\xi = U \in \mc U$, then for each $i$ we may write $a_i =
a_i'u_i^2$ for some $a_i' \in F^\times$ and $u_i \in F_U^\times$, by
Corollary~\ref{reduce_to_patches}(\ref{root_big}).   
Hence after adjusting $x_i$ by a factor of $u_i$, we may assume that each $a_i$ lies
in $F^\times$.  Now $u_s(K)=2u_s(k)$ by Theorem~4.10 of   
\cite{HHK}, 
where $K$ is the fraction field of~$T$; 
and so $n > 2u_s(K)$.  Since $F$ is a finitely generated field
extension of $K$ of transcendence degree one, it follows from the definition
of $u_s$ that the $n$-dimensional form $q$ is isotropic over~$F$, and hence
over~$F_U$.

In the case that $\xi \in \mc P$,
Corollary~\ref{reduce_to_patches_II} yields a split cover $\wh X' \to \wh
X$ with function field $F'$ such that each $a_i \in F_P^\times = F_{P'}'^\times$ may be written as
$a_i'u_i^2$ for some $a_i' \in F'^\times$ and $u_i \in F_{P'}'^\times$, where $P'
\in \wh X'$ lies over $P$ (where we again identify $F_P$ with its trivial
extension $F_{P'}'$).   
Adjusting $x_i$ by a factor of $u_i \in F_{P'}'^\times = F_P^\times$, we may
thus assume that each $a_i$ lies in $F'^\times$.  
Since $F'$ is a finitely generated field extension of $K$ of transcendence
degree one, and since $u_s(K)=2u_s(k)$ as above, $q$ is isotropic over $F'$
and hence over $F_{P'}' = F_P$. 

(\ref{u-inv_lower_bound})  
First consider the case that $\xi = U \in \mc U$.   The local ring of $\tilde X$ at $\tilde Q$ is a discrete valuation ring with residue field $\kappa(\tilde Q)$ and fraction field equal to the function field $E$ of $U$.  Also, the localization $\wh R_{U,t}$ of $\wh R_U$ at its Jacobson radical is a discrete valuation ring with residue field $E$ and fraction field $F_U$.  Lemma~4.9 of \cite{HHK} then yields the inequalities $u(F_U) \ge  2u(E) \ge 4u(\kappa(\tilde Q))$, as asserted.

Next, consider the case $\xi = P \in \mc P$.  Let $\tilde \wp$ be the unique branch of $\tilde X$ at $\tilde P$, and let $\wp$ be the branch of $X$ at $P$
that $\tilde \wp$ lies over.  Let $\wh R_{P,\wp}$ be the localization of $\wh R_P$ at~$\wp$.  This is a discrete valuation ring with residue field
$\kappa(\wp)$ and fraction field $F_P$. 
The complete local ring of $\tilde X$ at $\tilde P$ is a discrete valuation ring with residue field $\kappa(\tilde P)$ and fraction field $\kappa(\wp') \cong \kappa(\wp)$.  Lemma~4.9 of \cite{HHK} then yields the desired inequalities $u(F_P) \ge  2u(\kappa(\wp)) \ge 4u(\kappa(\tilde P))$.
\end{proof}

Of course if the closed point $P$ (resp.\ $Q$) in Theorem~\ref{u-inv_patches}
(\ref{u-inv_lower_bound}) is a regular point of the closed fiber $X$, then we can simply consider the residue field at the point itself rather than passing to a point on the normalization $\tilde X$.

As a consequence, we obtain the following (compare Corollary~4.19 of
\cite{HHK}, and Question~4.11 of \cite{Hu}): 

\begin{cor} \label{u-inv_fin_exten_Laurent}
Let $k$ be a field of characteristic unequal to two 
and let $E$ be a
finite separable extension of $k((x,t))$.  
\renewcommand{\theenumi}{\alph{enumi}}
\renewcommand{\labelenumi}{(\alph{enumi})}
\begin{enumerate}
\item \label{u-upper_bound}
Then $u(E) \le 4u_s(k)$.
\item \label{u-value}
If $u(k)=u_s(k)$ and if every finite extension $k'$ of $k$ satisfies
$u(k')=u(k)$, then $u(E)=4u(k)$.   
\end{enumerate}
\end{cor}

\begin{proof}
Let $\wh X$ and $P$ be as in the conclusion of 
Lemma~\ref{local-to-global} applied to $E$, and extend the set $\{P\}$ 
to a finite 
subset $\mc P \subset X$ that contains all the points where
distinct components of $X$ meet.  Then $u(E) = u(F_P) \le 4u_s(k)$ 
by Theorem~\ref{u-inv_patches}(\ref{u-inv_upper_bound}), proving part
(\ref{u-upper_bound}).  For part (\ref{u-value}), $u(E) \le 4u_s(k) = 4u(k)$
by~(\ref{u-upper_bound}), while the reverse inequality is given by
Theorem~\ref{u-inv_patches}(\ref{u-inv_lower_bound}).
\end{proof}

Recall from \cite{Serre:CG}, II.4.5, that for $d>0$ a $C_d$-\textit{field} is
a field $k$ such that for every $m,n$ and every homogeneous polynomial $f$
over $k$ of degree $m$ in $n$ variables, $f$ has a non-trivial solution over
$k$ if $n>m^d$. 
As was observed after Corollary~4.12 of \cite{HHK}, if $k$ is a $C_d$ field
then $u_s(k) \le 2^d$.  
Hence Corollary~\ref{u-inv_fin_exten_Laurent}(\ref{u-upper_bound}) 
immediately gives:

\begin{cor} 
Let $k$ be a $C_d$ field of characteristic unequal to two.  Let $E$ be a
finite separable extension of $k((x,t))$.  Then $u(E) \le 2^{d+2}$. 
\end{cor}

More generally, Theorem~\ref{u-inv_patches}(\ref{u-inv_upper_bound}) implies: 
\begin{cor} \label{u-inv_cd}
Under the hypotheses of Theorem~\ref{u-inv_patches}, if $k$ is a $C_d$ field
then $u(F_\xi) \le 2^{d+2}$.   
\end{cor}

In the situation of this corollary, observe that if $u(\kappa(P))=2^d$
(resp.\ $u(\kappa(Q))=2^d$), then the opposite inequality also holds, by
Theorem~\ref{u-inv_patches}(\ref{u-inv_lower_bound}).  Hence in that case,
$u(F_\xi) = 2^{d+2}$.  In particular, if $k$ is algebraically closed
(resp.\ finite) and $\xi \in \mc P \cup \mc U$, then $u(F_\xi)= 4$
(resp.~$8$), generalizing
\cite[Corollary~4.18]{HHK}.  In the situation that $\xi = P \in \mc P$, these
special cases were shown in~\cite[Theorem~3.6]{COP} and
\cite[Theorem~1.5]{Hu}, respectively. 

We conclude this section with an example.

\begin{example} \label{exa_m-local_u-invariant}
Recall that $k$ is an  \textit{$m$-local field} with
\textit{residue field $k_0$} if there are fields $k_1,\dots,k_m$
with $k_m=k$, such that $k_i$ is the fraction field of an
excellent henselian discrete valuation ring with residue field $k_{i-1}$ for
$i=1,\dots,m$. For example, $k=k_0((z_1))((z_2))\cdots((z_m))$ is such a
field.

Now let $E$ be a finite separable extension of $k((x,t))$, where $k$ is an
$m$-local field whose residue field $k_0$ is algebraically closed
(resp.\ finite) of characteristic not $2$.  Then $u(E)$ is equal to $2^{m+2}$
(resp.\ $2^{m+3})$. 

To see this, first note by induction that 
$2^m u(k_0) \le u(k) \le u_s(k) = 2^m u_s(k_0)$, using \cite{HHK}, Lemma~4.9
and Corollary~4.12.   
If $k_0$ is algebraically closed (resp.\ finite), then $u(k_0)=u_s(k_0)$ and
so $u(k)=u_s(k) = 2^m u(k_0)$, which is equal to $2^m$ (resp.\ $2^{m+1}$).
Moreover the same holds for any finite extension $k'$ of $k$, since $k'$ is 
itself an $m$-local field whose residue field is a finite extension of $k_0$
(by induction).  Hence the hypotheses of
Corollary~\ref{u-inv_fin_exten_Laurent}(\ref{u-value}) are satisfied, and the
conclusion follows. 

As an explicit example, if $E$ is a finite separable
extension of $k_0((u))((x,t))$, then $u(E)$ equals $8$ (resp.\ $16$). 
\end{example}

\subsection{Applications to central simple algebras}\label{subsec_csa}

Corollaries~\ref{reduce_to_patches} and~\ref{reduce_to_patches_II} also have
applications to the period-index problem for Brauer groups.  Below we use
these corollaries to extend results from \cite{HHK}.    

Recall (e.g.\ from~\cite{Pie}) that the Brauer group $\Br(F)$ of a field $F$
consists of equivalence classes of central simple $F$-algebras.  The
\textit{period} of a central simple $F$-algebra $A$, or of its class $\alpha$, is
the order of the class in the Brauer group.  The \textit{index} of $A$ (or of
$\alpha)$ is the degree over $F$ of the division algebra that lies in the
class; or equivalently, the minimal degree of a field extension $L/F$ over
which $A$ splits (i.e.\ such that $A \otimes_F L$ is isomorphic to $\Mat_n(L)$
for some $n$).   
This is also the greatest common divisor of the degrees of the splitting
fields $L$ of $A$. The period always divides the index, and the index always
divides some power 
of the period.  The {\em period-index problem} asks for an exponent $d$
depending only on~$F$ such that all central simple
$F$-algebras~$A$ satisfy $\ind(A)|\;\per(A)^d$.  In asking this question, one
often restricts 
attention to elements whose period is not divisible by a certain prime number.

To help make this question more precise, there is the following terminology
(\cite{HHK}, Definition~1.3; see also \cite{Lieblich}).  The \textit{Brauer
  dimension} of a field $k$ 
(\textit{away from a prime~$p$}) is defined to be $0$ if $k$ is separably  
closed (resp.\ separably closed away from $p=\cha(k)$; i.e.\ the absolute
Galois group of $k$ is a pro-$p$ group).  Otherwise, it is defined to be the
smallest positive integer $d$ such that $\ind(A) | \per(A)^{d-1}$ for every
finite field extension $E/k$ and every central simple $E$-algebra $A$
(resp.\ with $p {\not |}\per(A)$), and also such that $\ind(A) | \per(A)^d$
for every finitely generated field extension $E/k$ of transcendence degree one
and every central simple $E$-algebra $A$  (resp.\ with $p {\not |}\per(A)$).
Note that the latter condition, on transcendence degree one extensions $E/k$,
is automatically satisfied in the case of fields $k$ that are separably closed
(resp.\ away from $p$), because in that situation $\Br(E)$ has no
non-trivial prime-to-$p$ torsion (see the paragraph before Proposition~5.2 of
\cite{HHK}).

Corollary~5.10 of \cite{HHK} related the period-index problem for a field $k$
to that of $F_P$ and $F_U$, where $k$ is the residue field of a complete
discrete valuation ring $T$, and $F_P$ and $F_U$ are the fields associated to
a point or open subset of the closed fiber of a smooth projective $T$-curve.  Using
Corollaries~\ref{reduce_to_patches} and~\ref{reduce_to_patches_II} above, that
result can now be generalized to the case of normal $T$-curves that are not
necessarily smooth.  
Specifically, we have the following result: 

\begin{thm} \label{per-ind-bound}
Let $T$ be a complete discrete valuation ring with residue field $k$ of
characteristic  $p \ge 0$. Let $\wh X$ be a normal projective $T$-curve with
closed fiber $X$, let $\mc P$ be a finite non-empty subset of $X$ that
contains all the points where distinct components of $X$ meet, and let $\mc U$
be the set of components of the complement of $\mc P$ in $X$.  Suppose that
$k$ has Brauer dimension $d$ away from $p$.  Then for every $\xi \in \mc P \cup \mc U$ and
for all $\alpha$ in $\Br(F_\xi)$ with period not divisible by $p$, we have
$\ind(\alpha)\, |\, \per(\alpha)^{d+2}$.  Moreover if $T$ contains a primitive $\per(\alpha)$-th root of unity, then $\ind(\alpha)\, | \,\per(\alpha)^{d+1}$. 
\end{thm}

\begin{proof}
Let $F$ be the function field of $\wh X$, let $\xi \in \mc P \cup \mc U$, and
consider $\alpha$ in $\Br(F_\xi)$ as in the assertion.  Observe that we may assume that for every $n>0$ that is not divisible by
$p$, and for every $a \in F_\xi$, there exist $a' \in F$ and $u
\in F_\xi^\times$ such that $a = a' u^n$.  Namely, if $\xi = U \in \mc U$ then
this property holds automatically by  
Corollary~\ref{reduce_to_patches}(\ref{root_big}); and if $\xi =  P \in \mc P$
then we can achieve this condition  
by replacing $\wh X$ by a split cover $\wh X'$ and replacing $F$ by the
function field $F'$ of $\wh X'$, by
Corollary~\ref{reduce_to_patches_II}.  (Again, $F'_{P'} = F_P$ in the notation
of Corollary~\ref{reduce_to_patches_II}, because $\wh X' \to \wh X$ is a split
cover.) 

The remainder of the proof is then the same as that of \cite{HHK},
Corollary~5.10, 
except that the observation in the above paragraph replaces the use of \cite{HHK}, Lemma~4.16.  For the convenience of the reader, we sketch the argument.

Writing the Brauer group as a direct product of its primary parts, we may
assume that $\per(\alpha)$ is a prime power $q^r$.
First consider the case that $F$ contains a primitive $q^r$-th root of unity.
Under this hypothesis, it follows from \cite{MerSus} that $\alpha$ can be
expressed as a tensor product    
$(a_1, b_1)_{q^r} \otimes \cdots \otimes~(a_m, b_m)_{q^r}$
of symbol algebras, 
where each $a_i,b_i \in F_{\xi}$.  Applying the observation in the first
paragraph of the proof, we may assume that each $a_i$ and $b_i$ lie in $F$.
By \cite{HHK}, Theorem~5.5 (or \cite{Lieblich}, Theorem~6.3), the fraction
field $K$ of $T$ has Brauer dimension at most $d+1$ away from $p$; and so
$\ind(\alpha)$ divides $\per(\alpha)^{d+1}$, completing the proof in this
case.  In the more general case, let $F' = F[\zeta_{q^r}]$; let $\wh X' = \wh
X \times_T T[\zeta_{q^r}]$ (which is \'etale over $\wh X$); and let $\alpha'$
be the induced element of $\Br(F'_{\xi'})$, for $\xi'$ on $\wh X'$ lying over
$\xi$ on $\wh X$.  Then $s:=[F'_{\xi'}:F_\xi(\zeta_q)]$ divides $q^{r-1}$;
$\ind(\alpha)$ divides $s\ind(\alpha')$; and $\per(\alpha')$ divides $q^r =
\per(\alpha)$.  So applying the previous case to $\alpha'$, it follows that
$\ind(\alpha)$ divides $\per(\alpha)^{d+2}$. 
\end{proof}

\begin{cor} \label{csa_laurent_extensions}
Let $k$ be a field of characteristic $p \ge 0$ having Brauer dimension $d$ away from~$p$, and let $E$
be a finite separable extension of $k((x,t))$.  If the period of $\alpha \in
\Br(E)$ is not divisible by $p$, and if $k$ contains a primitive $\per(\alpha)$-th root of unity, then $\ind(\alpha)\, | \,\per(\alpha)^{d+1}$. 
\end{cor}

\begin{proof}
Let $\wh X$ and $P$ be as given by Lemma~\ref{local-to-global} applied to $E$,
and choose a finite subset $\mc P$ of the closed fiber $X$ that contains $P$ and
the  points where distinct components of $X$ meet.  Taking $\xi = P$ in
Theorem~\ref{per-ind-bound} then gives the desired conclusion.
\end{proof}
 
\begin{example}
Let $k$ be an $m$-local field (see Example~\ref{exa_m-local_u-invariant}) 
whose residue field $k_0$ is separably closed away from~$p$, where $p :=
\cha(k_0) = \cha(k)\geq 0$. 
 If $E$
is a finite separable extension of $k((x,t))$, then $\ind(\alpha)$ divides
$\per(\alpha)^{m+1}$ for elements $\alpha \in \Br(E)$ for which $\per(\alpha)$ is not divisible by $p$.  
Namely, $k_0$ has Brauer dimension $0$ away from $p$; and so \cite{HHK}, Corollary~5.7, says that $k$ has Brauer dimension $m$ away from $p$.  Hence the assertion follows from Corollary~\ref{csa_laurent_extensions}.

Similarly, if $k$ is an $m$-local field of characteristic $p>0$ whose residue
field $k_0$ is finite, and if $k_0$ contains a
primitive $\per(\alpha)$-th root of unity, then $\ind(\alpha)$ divides
$\per(\alpha)^{m+2}$. 
\end{example}   

\begin{example}
Let $k = k_0(u)((z))$, where $k_0$ is separably closed.  If $E$
is a finite separable extension of $k((x,t))$, then $\ind(\alpha)$ divides
$\per(\alpha)^3$ in $\Br(E)$, provided $\per(\alpha)$ is not divisible by $p
:= \cha(k_0)\geq 0$.  To see this, note that any prime-to-$p$ finite extension of $k_0(u)$ has trivial Brauer group by Tsen's Theorem; and the period equals the index for elements of order not divisible by $p$ in the Brauer group of any one-variable function field over $k_0(u)$, by \cite{deJ}.  So $k_0(u)$ has Brauer dimension one away from $p$.  Thus $k$ has Brauer dimension at most two away from $p$, by \cite{HHK}, Corollary~5.6.  The assertion now follows from Corollary~\ref{csa_laurent_extensions}.

Similarly, if instead $k_0$ is a finite field containing a primitive
$\per(\alpha)$-th root of unity, then $\ind(\alpha)$ divides $\per(\alpha)^4$.
This follows by replacing \cite{deJ} by \cite{Lie}, which allows us to deduce
that $k_0(u)$ has Brauer dimension two away from $p := \cha(k_0)$ and hence
that $k$ has Brauer dimension three away from $p$.
\end{example}
  
Another consequence of Theorem~\ref{per-ind-bound} is the following generalization of \cite[Corollary~5.11]{HHK}, dropping
the smoothness assumption on $\wh X$ that was needed in the earlier result.

\begin{cor} \label{csa_patches_sc}
Under the hypotheses of Theorem~\ref{per-ind-bound}, if $k$ is separably
closed away from $p$, then $\per(\alpha)=\ind(\alpha)$ for elements 
in $\Br(F_\xi)$ of period not divisible by the characteristic of $k$.
\end{cor}

Here, as in Theorem~\ref{per-ind-bound}, $\xi$ can be in either $\mc P$ or
$\mc U$.  Combining this corollary in the former case with
Lemma~\ref{local-to-global} above, we obtain in 
particular:

\begin{cor}
Suppose that $k$ is separably closed away from $p = \cha(k)$.
If $E$ is a finite separable extension of $k((x,t))$, and if the period of
$\alpha \in \Br(E)$ is not divisible by $p$, then $\per(\alpha)=\ind(\alpha)$. 
\end{cor}

This can also be deduced from \cite{COP}, Theorem~2.1, where it was shown that
a class in $\Br(E)$ of period $n$ not divisible by $p$ represents a cyclic
algebra of index $n$, provided that $k$ is separably closed.  This immediately
gives the above assertion for such fields $k$.  More generally, if $k$ is
separably 
closed away from $p$, then the conclusion follows from \cite{COP},
Theorem~2.1, by taking the compositum of $E$ with the separable closure of
$k$, and using that the absolute Galois group of $k$ is a pro-$p$-group.

\bigskip

\noindent {\bf Author Information:}\\

\noindent David Harbater\\
Department of Mathematics, University of Pennsylvania, Philadelphia, PA 19104-6395, USA\\
email: harbater@math.upenn.edu
\medskip

\noindent Julia Hartmann\\
Lehrstuhl f\"ur Algebra, RWTH Aachen University, 52062 Aachen, Germany\\
email: hartmann@matha.rwth-achen.de

\medskip

\noindent Daniel Krashen\\
Department of Mathematics, University of Georgia, Athens, GA 30602, USA\\
email: dkrashen@math.uga.edu

\medskip

\noindent The first author was supported in part by NSF grant DMS-0901164. 
The second author was supported by the German Excellence Initiative via RWTH Aachen University and by the German National Science Foundation (DFG).
The third author was supported in part by NSA grant H98130-08-0109
and NSF grant DMS-1007462.

\end{document}